\documentclass[11pt,letterpaper]{article}
\usepackage[T1]{fontenc}
\usepackage{lmodern}
\usepackage[margin=1in]{geometry}
\usepackage{amsmath,amssymb,amsthm,mathtools,bm,mathrsfs}
\usepackage{cite}
\usepackage{enumitem,booktabs,microtype,array,placeins,graphicx,float}
\usepackage[hidelinks]{hyperref}
\usepackage{url}

\microtypesetup{expansion=false}
\allowdisplaybreaks
\newtheorem{theorem}{Theorem}[section]
\newtheorem{lemma}[theorem]{Lemma}
\newtheorem{proposition}[theorem]{Proposition}
\newtheorem{corollary}[theorem]{Corollary}
\theoremstyle{remark}

\numberwithin{equation}{section}

\floatstyle{ruled}
\newfloat{algorithm}{tbp}{loa}[section]
\floatname{algorithm}{Algorithm}

\newenvironment{keywords}{\par\small\noindent\textbf{Keywords. }}{\par}
\newenvironment{MSCcodes}{\par\small\noindent\textbf{MSC codes. }}{\par}

\newcommand{\R}{\mathbb R}
\newcommand{\N}{\mathbb N}

\newcommand{\supp}{\operatorname{supp}}

\newcommand{\JSR}{\widehat\rho}
\newcommand{\one}{\bm 1}

\newcommand{\eps}{\varepsilon}
\newcolumntype{P}[1]{>{\raggedright\arraybackslash}p{#1}}

\title{Chronological stability for nonstationary matrix refinement}

\author{Yuwen Li\quad and \quad Yupeng Wang%
\thanks{Corresponding author: Yupeng Wang. School of
Mathematical Sciences, Zhejiang University, Hangzhou, China
(e-mail: yupengw@zju.edu.cn).}}

\date{September 12, 2026}

\hypersetup{
  pdftitle={Chronological stability for nonstationary matrix refinement},
  pdfauthor={Yuwen Li and Yupeng Wang}
}

\begin{document}
\maketitle

\begin{abstract}
A matrix cascade is an iterative refinement process in which vector-valued data are repeatedly filtered by matrix-valued masks across successively finer scales. Such cascades arise in multichannel subdivision, Hermite refinement, multiresolution synthesis, and exponential-spline constructions, and in nonstationary schemes the masks may vary with scale. We develop a chronological stability framework for this setting, allowing both the transition operators and the associated matching spaces to change from level to level. If the chronological radius satisfies $\rho_{\rm ch}<1$ and the matching defects are summable, then the cascade converges uniformly and, for every $q\in(\rho_{\rm ch},1)$,
\[
\|F_n-\Phi\|_\infty
\lesssim
q^n+\sum_{j=1}^n\delta_j q^{\,n-j}
+\sum_{j>n}\delta_j .
\]
The time-weighted defect term records the level at which each perturbation is inserted, and the matching spaces need not converge; they may even alternate indefinitely. For compact piecewise-linear seeds, every finite ordered matrix cascade admits an exact fixed-width ReLU realization with $O(n)$ depth and parameters, and this linear depth order is worst-case optimal at fixed width. Consequently, exponentially decaying defects yield certified approximants of size $O(\log\varepsilon^{-1})$. We further construct stable two-channel classes in which the chronological radius and the leading metric-entropy coefficient are independent invariants, showing that stability does not determine finite-accuracy description complexity. The framework therefore provides tolerance-controlled finite representations for nonstationary multiresolution and exponential-spline refinement.
\end{abstract}


\section{Introduction}
\label{sec:iclr-introduction}

Subdivision and multiresolution synthesis generate curves, signals, and
refinable functions by repeated local filtering.  For a finitely supported
matrix mask $A^{[k]}=(A_j^{[k]})_{j\in J}$, one refinement level acts by
\[
 (V_k f)(x):=\sum_{j\in J} A_j^{[k]} f(2x-j),
 \qquad
 F_n:=V_1V_2\cdots V_n g .
\]
The stationary case has $A_j^{[k]}\equiv A_j$; in the nonstationary case the
mask varies with the scale $k$.  Such level dependence occurs in
scale-dependent multiresolution systems and Gaussian-based interpolatory
schemes \cite{CohenDyn1996,DynLevinYoon2007}.

Matrix masks arise naturally when several reproduced quantities are coupled.
For example, exponential-polynomial reproduction leads to nonstationary
subdivision rules used in geometric modeling \cite{JeongLeeYoon2013}, while
Hermite subdivision evolves vectors of function values and derivatives with
scale-dependent normalization
\cite{ContiCotroneiSauer2016,CotroneiEtAl2019}.  We encode the reproduced
modes by full-row-rank matching maps $P_k$ and write
\[
 K_k:=\ker P_k,
 \qquad
 P_kT_\epsilon^{[k]}=P_{k+1}+E_{k,\epsilon},
 \qquad
 \|E_{k,\epsilon}\|\lesssim\delta_k .
\]
Thus the relevant difference space $K_k$ is allowed to move with the level;
it need not converge to a fixed limiting subspace.  The notation above is
made precise in Section~\ref{sec:iclr-limit}.

The second difficulty is order.  A binary address generates the
noncommutative product
\[
 T_{b_1}^{[1]}T_{b_2}^{[2]}\cdots T_{b_n}^{[n]},
 \qquad b_s\in\{0,1\},
\]
so arbitrary permutations of the level matrices are not dynamically
admissible.  This motivates the chronological radius $\rho_{\rm ch}$,
which measures only consecutive, correctly ordered products.  The main
stability theorem proves, schematically,
\[
 \rho_{\rm ch}<1,
 \quad \sum_{k\ge1}\delta_k<\infty
 \quad\Longrightarrow\quad
 \|F_n-\Phi\|_\infty
 \lesssim
 q^n+\sum_{j=1}^n\delta_jq^{\,n-j}+\sum_{j>n}\delta_j,
 \qquad q>\rho_{\rm ch}.
\]
The factor $q^{n-j}$ retains the insertion time of the $j$th defect; reversing
or shuffling the factors changes this response.  This is the central reason
for keeping the prescribed clock instead of replacing it by an unrestricted
joint spectral-radius problem.

The computational task is therefore concrete: given the masks, a compact
seed $g$, and a target tolerance $\eps$, choose a finite depth $n$ from the
preceding bound and evaluate $F_n$ without tabulating all dyadic cells.  We
also construct an exact fixed-width ReLU realization of this finite ordered
cascade, so the analytical estimate becomes a certified stopping rule and a
finite representation.  The resulting generator can be reused in finite
signal-synthesis sums with known coefficients; Section~\ref{sec:applications}
gives the corresponding propagated error bound.  We do not infer masks from
observations, and convergence of a generator alone does not establish a
complete wavelet or perfect-reconstruction filter bank.

Stationary scalar and vector refinement connects convergence and regularity
to products of restricted transition matrices
\cite{CavarettaDahmenMicchelli1991,DaubechiesLagarias1991,
DaubechiesLagarias1992,HanJia1998,JiaRiemenschneiderZhou1998,
Han2003VectorCascade,Charina2012}.
Nonstationary vector convergence and approximate sum rules are treated in
\cite{CharinaConti2004,CharinaContiGuglielmiProtasov2017}.
The scalar matrix theory in the latter work already uses level-dependent
transition spaces and obtains a limiting space under convergence, with a
cluster-family JSR criterion. Our matrix result instead allows matching
kernels with no limit and tests only chronologically admissible products.
These are distinct extensions, not a claim that level-dependent spaces or
restricted spectral methods are new. Constrained-switching theory also
retains admissible products
\cite{Dai2012ConstrainedJSR,PhilippeEtAl2016Constrained}.
Here the admissibility and the kernel coordinates follow a prescribed,
possibly nonperiodic, level sequence.

For exact finite representation, Daubechies et al.
\cite{DaubechiesEtAl2023} treat iterates $V^n g$ of a single scalar refinement
operator and construct fixed-width, linear-depth ReLU realizations. Related
controller constructions appear in
\cite{Gantumur2026Atlas,BolorkhuuGantumur2026Loop,
Gantumur2026Tensor,BolorkhuuGantumur2026Affine}.
That stationary scalar construction is the starting point for our finite
compiler; the routing and seam-cancellation mechanism itself is not claimed as
new here. We extend it to arbitrary ordered products
$V_{\pi_1}\cdots V_{\pi_n}$ of level-dependent matrix masks, with architecture
bounds independent of the mask values and their order. More importantly, the
stability theorem in this paper is logically prior to, and independent of, the
compiler: it proves convergence and a quantitative tail estimate for
nonstationary matrix products with moving matching kernels and summable
matching defects. Thus the new analytical content is not an assumption of
cascade convergence but a criterion that produces it; the compiler then turns
that estimate into certified finite ReLU approximants.

Metric entropy measures the information required for uniform
reconstruction \cite{KolmogorovTikhomirov1959,BolcskeiEtAl2019}.
The exact hyperrectangle covering formula is known
\cite[Thm.~18]{AllardBolcskei2026Exact}. We realize that geometry inside a
uniformly stable matrix-cascade class, obtain optimal finite rational
network centers, and use the construction to distinguish the clock
statistics controlling stability from those controlling entropy.

\paragraph{Main contributions and comparison with prior work.}
The advances can be separated into four points.
\begin{enumerate}[label=(\roman*),leftmargin=2em,itemsep=3pt,topsep=3pt]
\item \emph{Chronological stability.}
Let $\rho_{\rm ch}$ be the chronological radius of the corrected restrictions
and let $\delta_k$ be the level-$k$ matching defect. We prove
\[
\rho_{\rm ch}<1,
\qquad
\sum_{k\ge1}\delta_k<\infty
\quad\Longrightarrow\quad
F_n=V_1\cdots V_ng\to\Phi
\quad\text{uniformly},
\]
with, for every $q\in(\rho_{\rm ch},1)$,
\begin{equation}
\|F_n-\Phi\|_\infty
\le C_q\left(
q^n+\sum_{j=1}^n\delta_jq^{\,n-j}
+\sum_{j>n}\delta_j
\right).
\label{eq:intro-tail}
\end{equation}
The convolution in \eqref{eq:intro-tail} retains the insertion time of each
defect. In contrast to an unrestricted JSR test, only products compatible with
the prescribed level order enter $\rho_{\rm ch}$.

\item \emph{Moving difference spaces.}
The criterion does not require $\ker P_k$ to converge to a limiting difference
space. Proposition~\ref{prop:matrix-moving-example} gives an exactly matched
matrix cascade with
\[
\delta_k=0,
\qquad
\rho_{\rm ch}=\frac12,
\qquad
F_n\to\Phi\not\equiv0,
\]
while the matching kernels alternate indefinitely. This is the structural
distinction from fixed-space stationary refinement and from reductions that
require a limiting matching geometry.

\item \emph{Exact realization of ordered matrix cascades.}
Extending the stationary scalar setting of \cite{DaubechiesEtAl2023}, every
ordered list of $n$ matrix masks and every fixed compact CPwL seed admit an
exact ReLU realization with
\[
\operatorname{width}=O(1),
\qquad
\operatorname{depth}=O(n),
\qquad
\#\operatorname{par}=O(n).
\]
The construction does not assume stability. For an explicit nonstationary
family with at least $2^n$ affine pieces, any width-$W$ exact ReLU realization
with $L$ hidden layers satisfies
\begin{equation}
(W+1)^L\ge2^n,
\qquad
L\ge\frac{n\log2}{\log(W+1)},
\label{eq:intro-depth}
\end{equation}
so the linear depth order is worst-case optimal. Combining this compiler with
\eqref{eq:intro-tail}, exponentially decaying defects give certified
$O(\log\eps^{-1})$-size approximants.

\item \emph{Stability versus description complexity.}
For an exactly matched two-channel class with $a\in(1/2,1)$,
\begin{equation}
H_\eps(\mathcal C_a)
=
\frac{(\log_2\eps^{-1})^2}{2\log_2a^{-1}}
+O(\log\eps^{-1}).
\label{eq:intro-entropy}
\end{equation}
Variable-rate examples show that the chronological radius and the leading
entropy coefficient depend on different statistics of the level sequence; in
particular, two clocks can have the same $\rho_{\rm ch}$ but different leading
metric entropies. This entropy separation and the associated optimal finite
codes are not part of the stationary exact-realization result.
\end{enumerate}

The rest of the paper is organized as follows.
Section~\ref{sec:iclr-limit} introduces the cascade notation and proves
the stability theorem, with examples of moving matching spaces.
Section~\ref{sec:iclr-exact} constructs exact networks and proves the
depth lower bound. Its application discussion in
Section~\ref{sec:applications} explains tolerance-controlled spline
synthesis and evaluation of known-generator signals.
Section~\ref{sec:iclr-bits} develops metric entropy and finite-accuracy
storage for fixed and variable contraction rates.
Section~\ref{sec:iclr-diagnostics} tests these constructions numerically
and discusses their scope.

\section{Chronological Stability}
\label{sec:iclr-limit}

\subsection{Matrix Cascades}
\label{sec:iclr-setting}

A continuous piecewise-linear function with finitely many breakpoints is
abbreviated CPwL. For vector functions,
$\|f\|_\infty=\sup_x\max_i|f_i(x)|$.
Matrix norms are induced norms, with the Euclidean norm indicated by a
subscript $2$. A ReLU network consists of affine maps separated by
coordinatewise ReLU activations, with an affine output layer. Width is the
largest hidden dimension, and the parameter count includes all affine
entries, including zeros.

Let $\rho(t)=\max\{t,0\}$ act componentwise.  Fix an output dimension $p$,
a finite index set $J\subset\mathbb Z$, and at level $k$ a matrix mask
$A^{[k]}=(A_j^{[k]})_{j\in J}$ with $A_j^{[k]}\in\R^{p\times p}$.  Define
\begin{equation}
 (V_kf)(x)=\sum_{j\in J}A_j^{[k]}f(2x-j),\qquad
 F_\pi=V_{\pi_1}\circ\cdots\circ V_{\pi_n}g .
\label{eq:iclr-cascade}
\end{equation}
The rightmost operator acts first.  This convention matters: the state below
contains the ordered product $T^{[\pi_1]}\cdots T^{[\pi_n]}$.

Choose integers $\ell_-<\ell_+$ containing $J$, $\supp g$, and $[0,1]$, with $g(\ell_-)=g(\ell_+)=0$,
and set $L=\ell_+-\ell_-$ and $D=pL$.  For $x\in[0,1]$, stack integer
translates in the block state
\[
 G_f(x)=\bigl(f(x+\ell_-),\ldots,f(x+\ell_+-1)\bigr)^T\in\R^D.
\]
Set $B_1(x)=\mathbf1_{[1/2,1]}(x)$,
$R(x)=2x-B_1(x)$, and $B_s(x)=B_1(R^{s-1}x)$. Define the
$D\times D$ block transition matrices by
\begin{equation}
 [T_\epsilon^{[k]}]_{i\ell}=
 A_{\ell_-+2i-\ell-1+\epsilon}^{[k]},
 \qquad \epsilon\in\{0,1\}.
\label{eq:iclr-transition}
\end{equation}
Out-of-support mask blocks are zero.  Direct reindexing gives
\begin{equation}
G_{F_\pi}(x)=T_{B_1(x)}^{[\pi_1]}\cdots
T_{B_n(x)}^{[\pi_n]}G_g(R^nx).
\label{eq:iclr-block}
\end{equation}
This identity is the bridge between refinement and a network state.  It also
isolates the difficulty: both $B_s$ and $R$ jump at dyadic seams.

For discrete data, the corresponding subdivision rule is
$(S_Ac)_i=\sum_j A_{i-2j}c_j$: upsampling followed by matrix filtering.
The mask changes this local rule, while the matching map below records the
reproduced modes. The tail bound will control the difference between a
finite-resolution output and its limiting signal.

\subsection{Convergence}

Let $C_{k,\alpha}:K\to K$ denote a uniformly bounded corrected stable
restriction at level $k$.  With $\chi_0=1$, define
\begin{equation}
 \chi_m:=\sup_{a\geq1}\ \sup_{\substack{\alpha_j\in I_j\\a\leq j<a+m}}
 \|C_{a,\alpha_a}\cdots C_{a+m-1,\alpha_{a+m-1}}\|,
 \qquad
 \rho_{\rm ch}:=\lim_{m\to\infty}\chi_m^{1/m}.
\label{eq:chronological-radius}
\end{equation}
The inner supremum varies the admissible transition at every level, whereas
the level indices remain consecutive and increasing.

\begin{proposition}
\label{prop:chronological-radius}
For all $m,\ell\ge0$,
\[
 \chi_{m+\ell}\le \chi_m\chi_\ell,
 \qquad
 \rho_{\rm ch}=\lim_{m\to\infty}\chi_m^{1/m}
 =\inf_{m\ge1}\chi_m^{1/m}.
\]
The value of $\rho_{\rm ch}$ is invariant under equivalent norms and uniformly
conditioned levelwise coordinate changes, and
\[
 \rho_{\rm ch}<1
 \iff (\exists m\ge1)\ \chi_m<1
 \iff (\exists C>0,\ 0<q<1)\ \chi_\ell\le Cq^\ell\ \forall \ell\ge0.
\]
If $\rho_{\rm ch}<1$, then for every $q\in(\rho_{\rm ch},1)$ there is
$C_q>0$ with $\chi_\ell\le C_q q^\ell$.  Moreover,
\[
 A=\operatorname{diag}(2,0),\quad B=\operatorname{diag}(0,2),\quad
 AB=BA=0
 \quad\Longrightarrow\quad
 \rho_{\rm ch}=0<\JSR\{A,B\}=2.
\]
\end{proposition}

The proof is in Section~\ref{sec:stability-proof}. The alternating
matrices above illustrate a distinction at the abstract cocycle level,
not a refinable-mask counterexample. Also, $\chi_m<1$ requires a bound
over every starting level. A finite prefix suffices only with additional
structure, such as known periodicity or a proved uniform tail enclosure.

\paragraph{Matching data}
Let the matrix masks be uniformly bounded with common finite support, let
$g\in C_c(\R;\R^p)$, and let $P_k\in\R^{r\times D}$, $1\leq r<D$, have full row rank with
$\sigma_{\min}(P_k)\geq\gamma>0$ and $\|P_k\|_2\leq\Gamma$.  If $U_k$ is an
orthonormal basis of $\ker P_k$, set
\[
 S_k=[\,U_k\ \ P_k^\dagger\,],\qquad
 \bar T_\epsilon^{[k]}=T_\epsilon^{[k]}-P_k^\dagger
 (P_kT_\epsilon^{[k]}-P_{k+1}).
\]
Let $K=\R^{D-r}\oplus0$ and let $\rho_{\rm ch}$ be the chronological radius of
\begin{equation}
 C_{k,\epsilon}:=
 \left.S_k^{-1}\bar T_\epsilon^{[k]}S_{k+1}\right|_K
 =U_k^TT_\epsilon^{[k]}U_{k+1}.
\label{eq:corrected-chronological-family}
\end{equation}
Define
\begin{equation}
 \delta_k=\max_{\epsilon\in\{0,1\}}
 \|(P_kT_\epsilon^{[k]}-P_{k+1})S_{k+1}\|_2
 +\sup_{t\in[0,1]}\|P_kG_{V_kg-g}(t)\|_2.
\label{eq:iclr-computable-defect}
\end{equation}

\begin{theorem}
\label{thm:iclr-computable-cascade}
Under the matching hypotheses above, assume
\[
 \rho_{\rm ch}<1,\qquad \sum_{k\ge1}\delta_k<\infty.
\]
Then $F_n:=V_1\cdots V_ng\to\Phi$ uniformly for some compactly supported
$\Phi$, and for every $q\in(\rho_{\rm ch},1)$,
\begin{equation}
 \|F_n-\Phi\|_\infty\leq C_q\left(q^n+
 \sum_{j=1}^n\delta_jq^{n-j}+\sum_{j>n}\delta_j\right).
\label{eq:iclr-computable-tail}
\end{equation}
Furthermore:
\[
 \delta_k\lesssim (k+1)^{-\alpha},\ \alpha>1
 \quad\Longrightarrow\quad
 \|F_n-\Phi\|_\infty\lesssim (n+1)^{1-\alpha},
\]
and, for CPwL $g$, fixed-width exact realizations achieve accuracy $\eps$ with
$O(\eps^{-1/(\alpha-1)})$ parameters.  If, in addition,
\[
 \delta_k\lesssim \theta^k,\ 0<\theta<1,
 \quad r\in(\max\{\rho_{\rm ch},\theta\},1)
 \quad\Longrightarrow\quad
 \|F_n-\Phi\|_\infty\lesssim r^n,
\]
with fixed-width exact realizations of size $O(n)$, hence
$O(\log\eps^{-1})$ at accuracy $\eps$.  With
\[
 \Lambda:=\max\left\{1,2\sup_k\sum_{j\in J}\|A_j^{[k]}\|_2\right\},
 \qquad
 \beta_r:=\frac{\log(r^{-1})}{\log(\Lambda/r)},
\]
one also has $\Phi\in C^{0,\beta_r}(\R;\R^p)$.
\end{theorem}

The pseudoinverse correction enforces
$P_k\bar T_\epsilon^{[k]}=P_{k+1}$, and the SVD controls the frame
conditioning, not the motion of $\ker P_k$.
Section~\ref{sec:stability-proof} proves the theorem directly from one
ordered-product lemma. The defect response is order-sensitive: for
$M_j=\left(\begin{smallmatrix}q&0\\\delta_j&1\end{smallmatrix}\right)$,
right extension of the product gives
\[
 M_1\cdots M_ne_1
 =q^ne_1+\Bigl(\sum_{j=1}^n\delta_jq^{n-j}\Bigr)e_2.
\]
Reversal changes the weight to $q^{j-1}$.

\begin{algorithm}[htbp]
\caption{Cascade truncation}
\label{alg:certified-cascade}
\noindent\textbf{Input.}
Mask and matching rules, a compact CPwL seed $g$, $\eps>0$,
and certified bounds for
Theorem~\ref{thm:iclr-computable-cascade}:
$q\in(\rho_{\rm ch},1)$, the tail constant $C_q$,
and a computable summable envelope $d_k\geq\delta_k$
with a certified tail.
\begin{enumerate}[leftmargin=1.8em,itemsep=3pt,topsep=2pt]
\item Choose $n$ such that
\[
 C_q\left(q^n+\sum_{j=1}^n d_jq^{n-j}+\sum_{j>n}d_j\right)\leq\eps.
\]
\item Compute $U_k,P_k^\dagger,S_k$ for $1\leq k\leq n+1$
and $T_\epsilon^{[k]},C_{k,\epsilon}$ for
$1\leq k\leq n$, $\epsilon\in\{0,1\}$.
Check the defect bounds, evaluating CPwL suprema at breakpoints.
\item Compile $\Phi_n=V_1\cdots V_ng$ using
Theorem~\ref{thm:iclr-exact} in the prescribed level order.
Return $\Phi_n$ with $\|\Phi_n-\Phi\|_\infty\leq\eps$.
\end{enumerate}
\end{algorithm}

The algorithm takes a uniform spectral certificate as input.
At fixed matrix dimension, processing $n$ levels uses $O(n)$
real-arithmetic matrix operations apart from certification. For a CPwL
seed, the residual suprema are attained at finitely many breakpoints.
Floating-point SVDs alone do not certify infinite-sequence bounds.

\FloatBarrier
\subsection{Proofs}
\label{sec:stability-proof}

\begin{proof}[Proof of Proposition~\ref{prop:chronological-radius}]
Splitting a product gives $\chi_{m+\ell}\leq\chi_m\chi_\ell$.
The submultiplicative limit is therefore
$\rho_{\rm ch}=\inf_m\chi_m^{1/m}$. Equivalent norms change $\chi_m$ by
a fixed factor. A uniformly conditioned coordinate change
$C'_{k,\alpha}=R_k^{-1}C_{k,\alpha}R_{k+1}$ telescopes along each product,
again changing its norm by at most a fixed factor. Taking roots proves
both invariances.

For $q>\rho_{\rm ch}$ choose $s$ with $\chi_s<q^s$. Writing $m=us+v$,
$0\leq v<s$, gives
\[
 \chi_m\leq\chi_s^u\chi_v\leq
 q^m\max_{0\leq v<s}q^{-v}\chi_v.
\]
This proves the exponential bound and its equivalence to $\chi_s<1$.
For a uniformly bounded finite-dimensional family, the distance of the
level-$k$ matrices to their cluster family tends to zero. If that family
has JSR below one, an adapted norm contracts it, and therefore all
sufficiently late matrices, by a factor below one. The finite prefix
affects only the prefactor. Finally, alternating $A$ and $B$ have zero
products of length two, while the unrestricted family contains $A^m$
with norm $2^m$. This gives the asserted separation.
\end{proof}

\begin{lemma}
\label{lem:ordered-perturbation}
Let $X=K\oplus E$ and let $M_{k,\alpha},\bar M_{k,\alpha}$ be uniformly bounded with
\[
 \bar M_{k,\alpha}=
 \begin{pmatrix}C_{k,\alpha}&B_{k,\alpha}\\0&I\end{pmatrix},
 \qquad
 \|M_{k,\alpha}-\bar M_{k,\alpha}\|\le\delta_k,
 \qquad
 \sum_k\delta_k<\infty.
\]
If the chronological radius of $\{C_{k,\alpha}\}$ satisfies
$\rho_{\rm ch}<1$, then
\[
 \sup_{a\le b}\ \sup_{\alpha_a,\ldots,\alpha_b}
 \|P_{a,b}\|<\infty,
 \qquad
 P_{a,b}:=M_{a,\alpha_a}\cdots M_{b,\alpha_b},
\]
and, for every $q\in(\rho_{\rm ch},1)$,
\begin{equation}
 \|P_{a,b}(v+w)\|
 \leq C_q\left[
 \left(q^{b-a+1}+\sum_{j=a}^b\delta_jq^{b-j}\right)\|v\|+\|w\|
 \right],
 \quad v\in K,\ w\in E.
\label{eq:ordered-perturbation}
\end{equation}
\end{lemma}

\begin{proof}
Suppress the selection indices and write
$\bar P_{a,b}=\bar M_a\cdots\bar M_b$.
The stable block is $O(q^{b-a+1})$, and the upper-right block
$B_a+C_aB_{a+1}+\cdots+C_a\cdots C_{b-1}B_b$
is bounded by a geometric series. Thus
$\|\bar P_{a,b}(v+w)\|\leq C_q(q^{b-a+1}\|v\|+\|w\|)$.
With empty products equal to the identity, the exact ordered identity is
\begin{equation}
 P_{a,b}=\bar P_{a,b}
 +\sum_{j=a}^b P_{a,j-1}(M_j-\bar M_j)\bar P_{j+1,b}.
\label{eq:ordered-duhamel}
\end{equation}
It implies
$\|P_{a,b}\|\leq C+C\sum_{j=a}^b\delta_j\|P_{a,j-1}\|$.
Discrete Gronwall and summability give a bound uniform in $a,b$ and the
selections. Applying \eqref{eq:ordered-duhamel} to $v+w$ and using the
stable bound on its corrected suffix proves
\eqref{eq:ordered-perturbation}.
\end{proof}

\begin{proof}[Proof of Theorem~\ref{thm:iclr-computable-cascade}]
Orthogonality of $\ker P_k$ and $\operatorname{range}(P_k^T)$ gives
\[
 S_k^{-1}=\begin{bmatrix}U_k^T\\P_k\end{bmatrix},\qquad
 \|S_k\|_2=\max\{1,\sigma_{\min}(P_k)^{-1}\},\qquad
 \|S_k^{-1}\|_2=\max\{1,\|P_k\|_2\}.
\]
Thus both frame norms are uniform. Set
$E_{k,\epsilon}=P_kT_\epsilon^{[k]}-P_{k+1}$.
Since $P_kP_k^\dagger=I$ and $U_k^TP_k^\dagger=0$,
\[
 P_k\bar T_\epsilon^{[k]}=P_{k+1},\qquad
 S_k^{-1}\bar T_\epsilon^{[k]}S_{k+1}
 =\begin{pmatrix}
 C_{k,\epsilon}&U_k^TT_\epsilon^{[k]}P_{k+1}^\dagger\\0&I
 \end{pmatrix},
\]
and
\[
 S_k^{-1}(T_\epsilon^{[k]}-\bar T_\epsilon^{[k]})S_{k+1}
 =\begin{bmatrix}0\\E_{k,\epsilon}S_{k+1}\end{bmatrix}.
\]
The pulled-back transitions therefore satisfy
Lemma~\ref{lem:ordered-perturbation}. Frame factors telescope in their
chronological products.

Set $h_{n+1}=V_{n+1}g-g$ and split
$S_{n+1}^{-1}G_{h_{n+1}}=v_{n+1}+w_{n+1}$ along
$K=\mathbb R^{D-r}\oplus0$ and its coordinate complement.
The common support window and uniform mask and frame bounds give
$\|v_{n+1}\|_\infty\leq C$, while
\eqref{eq:iclr-computable-defect} gives
$\|w_{n+1}\|_\infty\leq\delta_{n+1}$.
The block identity \eqref{eq:iclr-block} yields
\[
 G_{F_{n+1}-F_n}(t)=T_{B_1(t)}^{[1]}\cdots
 T_{B_n(t)}^{[n]}G_{h_{n+1}}(R^nt).
\]
Applying the lemma and restoring the outer frame gives
\begin{equation}
 \|F_{n+1}-F_n\|_\infty
 \leq C_q\left(q^n+\sum_{j=1}^n\delta_jq^{n-j}
                         +\delta_{n+1}\right).
\label{eq:compact-increment}
\end{equation}
The block norm controls the global norm because all cascades remain in
the same compact support window. Moreover,
\[
 \sum_{n\geq0}\sum_{j=1}^n\delta_jq^{n-j}
 =\frac1{1-q}\sum_{j\geq1}\delta_j<\infty.
\]
Hence the cascade is uniformly Cauchy, with a continuous compactly
supported limit. Summing \eqref{eq:compact-increment} from $n$ onwards
and splitting the defect index at $n$ proves
\eqref{eq:iclr-computable-tail}.

For $\delta_k=O((k+1)^{-\alpha})$, split the convolution at
$\lfloor n/2\rfloor$. Its two parts are $O(q^{n/2})$ and
$O((n+1)^{-\alpha})$, while the future tail is
$O((n+1)^{1-\alpha})$. For $\delta_k=O(\theta^k)$, choose
$\rho_{\rm ch}<q<r$ with $\theta<r<1$. The convolution and future tail
are $O(r^n)$. The exact compiler in Theorem~\ref{thm:iclr-exact}
then gives both parameter estimates.

Finally, for CPwL $g$,
$\operatorname{Lip}(F_n)\leq\operatorname{Lip}(g)\Lambda^n$.
For $h=|x-y|\leq1$,
\[
 \|\Phi(x)-\Phi(y)\|_2
 \leq2C_rr^n+\operatorname{Lip}(g)\Lambda^nh.
\]
Choose $n$ with $(r/\Lambda)^{n+1}<h\leq(r/\Lambda)^n$.
Both terms are $O(h^{\beta_r})$, with $\beta_r$ as stated.
Boundedness handles $h>1$, proving the H\"older conclusion.
\end{proof}

\subsection{Examples}

\begin{proposition}
\label{prop:matrix-moving-example}
Let
\[
 a=(1/2,1,1/2),\qquad
 s_k=\mathbf 1_{2\mid k},\qquad
 C_k=\begin{pmatrix}1&0\\s_k&1\end{pmatrix},\qquad
 B_j=\begin{pmatrix}0&0\\0&a_j\end{pmatrix},\qquad
 A_j^{[k]}=C_kB_jC_{k+1}^{-1}.
\]
There exists a compactly supported CPwL seed $g$ such that
\[
 \delta_k=0,\qquad \rho_{\rm ch}=\frac12,\qquad
 V_1\cdots V_ng\to\Phi\not\equiv0,
\]
while the matching kernels alternate between two distinct subspaces.
\end{proposition}

\begin{proof}
Let $h(t)=\max\{1-|t-1|,0\}$ and define
\[
 u(t)=
 \begin{cases}t/2,&0\leq t\leq1/2,\\
 3t/2-1/2,&1/2\leq t\leq1,\end{cases}
 \qquad
 \psi(t)=
 \begin{cases}u(t),&0\leq t\leq1,\\
 1-u(t-1),&1\leq t\leq2,\\0,&\text{otherwise}.
 \end{cases}
\]
Take $g=\psi e_2$. If $\mathcal T_a$ is scalar refinement by $a$ and $W$
is vector refinement by $B$, then
\[
 V_k=C_kWC_{k+1}^{-1},\qquad
 V_1\cdots V_ng=(\mathcal T_a^n\psi)e_2.
\]
Indeed, every $C_k^{\pm1}$ fixes $e_2$ and $C_1=I$.
Both $\psi$ and $h$ have partition of unity. Their block difference on
$[0,1]$ is a multiple of $(1,-1)^T$, on which both scalar transitions
act by $1/2$. Since $\mathcal T_ah=h$,
$\|\mathcal T_a^n\psi-h\|_\infty\leq
2^{-n}\|\psi-h\|_\infty$.

On the window $[0,2]$, set
$P_0=(0,1,0,1)$,
$\mathcal S_k=I_2\otimes C_k$, and
$P_k=P_0\mathcal S_k^{-1}=(-s_k,1,-s_k,1)$.
The identity
$T_\epsilon^{[k]}=\mathcal S_kT_{\epsilon,B}\mathcal S_{k+1}^{-1}$
gives $P_kT_\epsilon^{[k]}=P_{k+1}$.
Partition of unity gives the zero seed defect.
The matrices $P_k$ have singular values $\sqrt2$ or $2$, and the shears
and their inverses are uniformly bounded. On $\ker P_0$, the first channel
is annihilated and the second-channel difference contracts by $1/2$.
Consequently the chronological radius is $1/2$, also in the orthonormal
matching coordinates. Finally,
$\ker P_k=\mathcal S_k\ker P_0$ alternates, with successive orthogonal
projectors at distance $1/\sqrt2$.
\end{proof}

\label{sec:exponential-diagnostic}

\begin{proposition}
\label{prop:exponential-spline-certificate}
Fix $\lambda_1,\lambda_2\in\R$ and, for $\ell=1,2$, $k\ge1$, set
\[
 r_{\ell,k}=e^{\lambda_\ell2^{-k-1}},\qquad
 a_{\ell,j}^{[k]}=\frac{2\binom2j r_{\ell,k}^j}{(1+r_{\ell,k})^2},
 \quad
 A_j^{[k]}=\operatorname{diag}(a_{1,j}^{[k]},a_{2,j}^{[k]}),
 \quad j=0,1,2,
\]
and
\[
 s_{\ell,k}=e^{\lambda_\ell2^{-k}},\qquad
 P_k=
 \begin{pmatrix}
 (1+s_{1,k}^2)^{-1/2}&0&s_{1,k}(1+s_{1,k}^2)^{-1/2}&0\\
 0&(1+s_{2,k}^2)^{-1/2}&0&s_{2,k}(1+s_{2,k}^2)^{-1/2}
 \end{pmatrix}.
\]
For $g=(h,h)^T$, $h(x)=(1-|x-1|)_+$,
\[
 P_kP_k^T=I_2,\qquad
 \delta_k=O(2^{-k}),\qquad
 \rho_{\rm ch}\le\frac12,
\]
and the corrected cluster JSR equals $1/2$.
and, for every $q\in(1/2,1)$,
\[
 \|V_1\cdots V_ng-\Phi\|_\infty=O_q(q^n).
\]
\end{proposition}

\begin{proof}
The scalar symbols are normalized order-two exponential-spline factors
$2((1+r_{\ell,k}z)/(1+r_{\ell,k}))^2$
\cite{JeongLeeYoon2013}. The masks approach $(1/2,1,1/2)$ at rate
$O(2^{-k})$. The rows of $P_k$ are orthonormal, so
$S_k=[U_k,P_k^T]$ is orthogonal. At the common limit,
$P_\infty T_{\epsilon,\infty}=P_\infty$.
Smooth dependence on the exponential parameters therefore gives
$\|P_kT_\epsilon^{[k]}-P_{k+1}\|_2=O(2^{-k})$.
Choose a smoothly varying orthonormal kernel basis. The two corrected
restricted transitions then approach $\frac12 I_2$, giving cluster JSR
$1/2$. The chronological conclusion follows from
Proposition~\ref{prop:chronological-radius}.
Since the limiting mask exactly refines $h$,
$\|V_kg-g\|_\infty=O(2^{-k})$ as well.
Theorem~\ref{thm:iclr-computable-cascade} proves the tail estimate.
\end{proof}

\section{Exact Realization}
\label{sec:iclr-exact}

This result is algebraic and does not assume convergence, matching, or a
spectral gap. Its architecture is uniform over prescribed mask lists,
whereas its weights may depend on those lists.

\begin{theorem}
\label{thm:iclr-exact}
Fix $p$, $J$, and a compactly supported continuous CPwL seed
$g:\R\to\R^p$.  There exist constants $W,C>0$, depending only on
$(p,J,g)$, such that for every ordered list $\pi$ of $n$ masks supported in
$J$ there is a one-input, $p$-output ReLU network $\mathcal N_\pi$ with
\[
 \mathcal N_\pi=F_\pi\ \text{on }\R,\qquad
 \operatorname{width}(\mathcal N_\pi)\le W,
 \qquad
 \operatorname{depth}(\mathcal N_\pi)\le Cn,
 \qquad
 \#\operatorname{par}(\mathcal N_\pi)\le Cn.
\]
\end{theorem}

\begin{proposition}
\label{thm:iclr-depth-lower}
There exist a compactly supported CPwL seed $g$ and bounded scalar masks
$a^{[k]}$ supported in $\{0,1\}$ such that
\[
 \#\operatorname{pieces}(F_n)\ge2^n.
\]
Hence every width-$W$, depth-$L$ ReLU network representing $F_n$ exactly
satisfies
\begin{equation}
 (W+1)^L\geq 2^n,
 \qquad
 L\geq \frac{n\log 2}{\log(W+1)}.
\label{eq:iclr-depth-lower}
\end{equation}
Thus fixed-width exact realization requires $L=\Omega(n)$ and
$\#\operatorname{par}=\Omega(n)$.
\end{proposition}

The construction has two steps.
A continuous router computes the terminal seed while canceling ambiguous
dyadic seams. An exact selector then updates a fixed-dimensional dual
state using $(T_\epsilon^{[\pi_s]})^T$ at step $s$. The identity
$(T_s)^T\cdots(T_1)^T=(T_1\cdots T_s)^T$
preserves the original cascade order. A fixed hat decomposition and
integer-cell patching handle a general compact CPwL seed.
Thus the construction executes the refinement rules rather than storing
every affine piece. Section~\ref{sec:obstruction} proves the lower bound.

\subsection{Routing}
\label{sec:exact-realization}

Consider first a localized atom
\begin{equation}\label{eq:special-atom}
 g(x)=h(x)v,\qquad h\geq0,\qquad
 \supp h\subset[1/8,7/8],\qquad v\in\R^p,
\end{equation}
where $h$ is CPwL. Set $i_0:=1-\ell_-$. Since the block window contains
$[0,1]$, one has $1\leq i_0\leq L$, and
\[
 G_g(y)=h(y)q,\qquad
 q=e_{i_0}^{(L)}\otimes v\in\R^D,
 \qquad y\in[0,1].
\]

The maps $B_s$ are discontinuous, while every ReLU network is continuous.
The following controller makes that cancellation explicit. It is the
mask-independent part of the stationary construction in
\cite{DaubechiesEtAl2023}, written here in
a form that accepts level-dependent transition matrices.

\begin{lemma}
\label{lem:controller}
For the atom in \eqref{eq:special-atom} and every $n\ge1$, there is a
mask-independent fixed-width, $O(n)$-depth ReLU construction of continuous
maps $\widehat R_n,\widehat\chi_{0,n},\widehat\chi_{1,n}:[0,1]\to[0,1]$,
a set $E_n\subset[0,1]$, and $H_n$ such that
\[
 H_n=h\circ R^n,\qquad H_n|_{E_n}=0,
\]
and, for $x\notin E_n$, $1\le s\le n$, $\eps\in\{0,1\}$,
\[
 \widehat R_n^{\,s-1}(x)=R^{s-1}(x),\qquad
 \widehat\chi_{\eps,n}(\widehat R_n^{\,s-1}(x))
 =\mathbf1_{\{B_s(x)=\eps\}}.
\]
The construction also propagates $x$ exactly to the selector stage.
\end{lemma}

\begin{proof}
Let $d_n:=2^{-n-5}$ and define
\[
 \widehat\chi_{1,n}(t):=
 \frac{\rho(t-\tfrac12+d_n)-\rho(t-\tfrac12-d_n)}{2d_n},
 \qquad
 \widehat\chi_{0,n}:=1-\widehat\chi_{1,n},
\]
and
\[
 \widehat R_n(t):=2t-\widehat\chi_{1,n}(t).
\]
The selector is zero to the left of $1/2-d_n$, one to the right of
$1/2+d_n$, and affine in between. Consequently, $\widehat R_n$ agrees
with $R$ outside that transition interval and maps $[0,1]$ into itself.
Define
\begin{equation}\label{eq:seam-set}
 E_n:=\left\{x\in[0,1]:
 |R^j(x)-\tfrac12|<d_n\ \text{for some }0\leq j<n\right\}.
\end{equation}
If $x\notin E_n$, induction on $j$ gives
\begin{equation}\label{eq:exact-continuous-orbit}
 \widehat R_n^j(x)=R^j(x),\qquad
 \widehat\chi_{\eps,n}(\widehat R_n^{j-1}(x))
 =\mathbf 1_{\{B_j(x)=\eps\}},
 \quad 1\leq j\leq n.
\end{equation}

We next construct an exact terminal factor. If $x\in E_n$, then for some
$0\leq j<n$ and an integer $m$,
\[
 |2^jx-m-\tfrac12|<d_n.
\]
The dyadic point $(2m+1)2^{-j-1}$ belongs to $2^{-n}\mathbb Z$, so
\begin{equation}\label{eq:seam-distance}
 \operatorname{dist}(x,2^{-n}\mathbb Z)<2^{-j}d_n\leq d_n.
\end{equation}
Let
\[
 \tau(t):=2\rho(t)-4\rho(t-\tfrac12)+2\rho(t-1).
\]
For $t\in[0,1]$, this is the tent map. Its iterates satisfy
\begin{equation}\label{eq:tent-distance}
 \tau^{n+1}(t)=2^{n+1}\operatorname{dist}(t,2^{-n}\mathbb Z).
\end{equation}
For $n=0$, the identity is the formula for the tent itself. For the
induction step, write $u=2^n\operatorname{dist}(t,2^{-n}\mathbb Z)$,
so that $0\leq u\leq1/2$. Then
$\tau(2u)=4\min\{u,1/2-u\}$, and the minimum equals
$2^n\operatorname{dist}(t,2^{-(n+1)}\mathbb Z)$. This gives the next
iterate in \eqref{eq:tent-distance}.
Finally, set
\[
 \zeta(u):=16\rho(u-\tfrac1{16})-16\rho(u-\tfrac18),
 \qquad q_n(t):=\zeta(\tau^{n+1}(t)),
\]
and, with $M_h:=\max\{1,\|h\|_\infty\}$, define
\[
 P_M(s,y):=-\rho(Ms-y)-\rho\bigl(M(1-s)-\rho(-y)\bigr)+M
\]
and
\begin{equation}\label{eq:terminal-cutoff}
 H_n(t):=P_{M_h}\bigl(q_n(t),h(\widehat R_n^n(t))\bigr).
\end{equation}
A direct three-case calculation gives, for $|y|\leq M$ and $s\in[0,1]$,
\[
 P_M(0,y)=0,\qquad P_M(1,y)=y,\qquad P_M(s,0)=0.
\]

The distance identity linking the terminal residual to dyadic seams is
\[
 \operatorname{dist}(R^nt,\{0,1\})
 =2^n\operatorname{dist}(t,2^{-n}\mathbb Z)
 =\tfrac12\tau^{n+1}(t),\qquad t\in[0,1].
\]
It holds at dyadic endpoints as well, under the stated binary convention.
We prove $H_n(t)=h(R^n(t))$ by cases. If $t\in E_n$, then
\eqref{eq:seam-distance}--\eqref{eq:tent-distance} give
$\tau^{n+1}(t)<1/16$, hence $q_n(t)=0$. Moreover, $R^n(t)$ is within
$1/32$ of an endpoint of $[0,1]$, so the support assumption on $h$ gives
$h(R^n(t))=0$. The identity $P_{M_h}(0,y)=0$ proves the claim. If
$t\notin E_n$, equation \eqref{eq:exact-continuous-orbit} gives
$\widehat R_n^n(t)=R^n(t)$. When $q_n(t)=1$, the identity
$P_{M_h}(1,y)=y$ applies. When $q_n(t)<1$, one has
$\tau^{n+1}(t)<1/8$, so $R^n(t)$ is within $1/16$ of an endpoint and
$h(R^n(t))=0$; now $P_{M_h}(s,0)=0$ applies. Thus
$H_n=h\circ R^n$ everywhere, and in particular it vanishes on $E_n$.

Each one-step scalar map displayed above is CPwL with a fixed number of
pieces. Iterating
$\widehat R_n$ and $\tau$ costs $O(n)$ layers at fixed width; the fixed
CPwL function $h$ and the terminal gate add $O(1)$ layers. A signed copy of
the input is propagated through these layers using
$t=\rho(t)-\rho(-t)$. This proves all assertions.
\end{proof}

The selectors need only choose between two states when their value is zero
or one. At seams the state is already zero. This permits an exact ReLU gate
without multiplication.

\begin{lemma}
\label{lem:product-gate}
For $M>0$, $s\in[0,1]$, and $y\in\R^D$ with $\|y\|_\infty\le M$, define
\begin{equation}\label{eq:product-gate}
 \Pi_M(s,y):=-\rho(Ms\one-y)
 -\rho\bigl(M(1-s)\one-\rho(-y)\bigr)+M\one.
\end{equation}
Then $\Pi_M$ is realizable by a fixed-size ReLU module and
\[
 \Pi_M(1,y)=y,\qquad \Pi_M(0,y)=0,\qquad \Pi_M(s,0)=0.
\]
\end{lemma}

\begin{proof}
The verification is coordinatewise. If $s=1$, then
$\rho(M-y_i)=M-y_i$ because $|y_i|\leq M$, while the second ReLU term is
zero. Hence the result is $y_i$. If $s=0$ and $y_i\geq0$, the two terms are
$0$ and $M$. If $s=0$ and $y_i<0$, they are $-y_i$ and $M+y_i$. Both
cases give zero after adding $M$. When $y_i=0$, the two ReLU terms are
$Ms$ and $M(1-s)$, which cancel the last term. The displayed formula uses
two ReLU layers and a number of units depending only on $D$.
\end{proof}

\subsection{State Recursion}

Fix an output coordinate $r\in\{1,\ldots,D\}$ and initialize
\[
 z_r^0(x):=H_n(x)e_r.
\]
The original input is still available after this initialization.
Set $u_0(x):=x$ and update $u_s=\widehat R_n(u_{s-1})$ in parallel with
the matrix state. Thus $u_{s-1}=\widehat R_n^{s-1}(x)$ is available at
step $s$ without retaining an entire digit sequence.
For $s=1,\ldots,n$, set
\begin{equation}\label{eq:state-recursion}
\begin{aligned}
 z_r^s={}&\Pi_M\left(
 \widehat\chi_{0,n}(\widehat R_n^{s-1}(x)),
 (T_0^{[\pi_s]})^Tz_r^{s-1}\right)\\
 &+\Pi_M\left(
 \widehat\chi_{1,n}(\widehat R_n^{s-1}(x)),
 (T_1^{[\pi_s]})^Tz_r^{s-1}\right).
\end{aligned}
\end{equation}
For a fixed finite list, all ideal intermediate states have finite norm, so
one may choose a common $M$ large enough for every gate. This choice changes
the weights but not the architecture.

Away from $E_n$, exactly one selector is one and the other is zero.
Lemma~\ref{lem:product-gate} therefore selects the corresponding matrix.
On $E_n$, the initial state is zero and the identity
$\Pi_M(s,0)=0$ keeps every later state zero. Induction gives, for all
$x\in[0,1]$,
\begin{equation}\label{eq:state-product}
 z_r^s(x)=h(R^n(x))
 \left(T_{B_1(x)}^{[\pi_1]}\cdots
 T_{B_s(x)}^{[\pi_s]}\right)^Te_r.
\end{equation}
The induction also confirms the order. At the second step,
$(T_2)^T(T_1)^T=(T_1T_2)^T$, so the recursion produces the product in
\eqref{eq:iclr-block}, not its reversal. Since $q^Tz_r^n$ is the
$r$th coordinate of the right-hand side of \eqref{eq:iclr-block}, all
coordinates of the block state are recovered exactly.

The dimension $D$ is fixed. Each update adds only fixed width and a fixed
number of layers. The terminal controller has depth $O(n)$ and the $n$ state
updates have depth $O(n)$. Running the $D$ coordinate recursions in parallel
preserves fixed width. If hidden states are required to be nonnegative,
each signed coordinate is stored as the difference of two nonnegative
channels, which changes the width only by a fixed factor.

\subsection{General Seeds}

For the special atom \eqref{eq:special-atom}, first extend the construction
from $[0,1]$ to all of $\R$. Since $h$ is supported in
$[1/8,7/8]$, it vanishes at every integer. If a function $f$ vanishes at
all integers, then
\[
 (V_kf)(m)=\sum_{j\in J}A_j^{[k]}f(2m-j)=0,
 \qquad m\in\mathbb Z.
\]
Every finite cascade of the special atom therefore vanishes at all integer
seams. Define
\[
 r_i(x):=\rho(x-\ell_--i+1)-\rho(x-\ell_--i),
 \qquad i=1,\ldots,L.
\]
On $[\ell_-+i-1,\ell_-+i]$,
$r_i(x)=x-\ell_--i+1$, while outside that interval it equals zero or one.
Because every block piece vanishes at both endpoints,
\[
 F_\pi(x)=\sum_{i=1}^L
 [G_{F_\pi}]_i(r_i(x)),\qquad x\in\R.
\]
All nonmatching pieces vanish automatically. Since $L$ is fixed, this
patching adds only fixed width and a fixed number of layers.

Now let $g$ be general. Take a common finite nodal partition containing the
breakpoints of every component of $g$, and refine it until adjacent nodes
are less than $3/8$ apart. Its nonnegative nodal hat functions
$\lambda_1,\ldots,\lambda_Q$ give
\begin{equation}\label{eq:hat-decomposition}
 g(x)=\sum_{\nu=1}^Q\lambda_\nu(x)v_\nu,
 \qquad v_\nu\in\R^p,
\end{equation}
where $Q$ is independent of $n$. Each hat has support of length less than
$3/4$, so there is a translation $\delta_\nu$ and a special atom
$h_\nu$ supported in $[1/8,7/8]$ such that
$\lambda_\nu=\tau_{\delta_\nu}h_\nu$, with
$(\tau_\delta f)(x)=f(x-\delta)$. A direct substitution into
\eqref{eq:iclr-cascade} shows
\begin{equation}\label{eq:translation-commutation}
 V_k\tau_\delta=\tau_{\delta/2}V_k.
\end{equation}
After $n$ levels, the cascade of $\lambda_\nu v_\nu$ is therefore the
already constructed global special-atom network with its input translated
by $\delta_\nu/2^n$. Run the fixed number $Q$ of networks in parallel and
sum their outputs according to \eqref{eq:hat-decomposition}. The resulting
network is exact, has fixed width, and has depth $O(n)$. This completes the
construction.

\begin{proof}[Proof of Theorem~\ref{thm:iclr-exact}]
The block identity \eqref{eq:iclr-block} writes every block of the cascade as the
ordered transition product in \eqref{eq:iclr-block}.
Lemma~\ref{lem:controller} supplies continuous residuals and selectors, and
Lemma~\ref{lem:product-gate} implements each binary matrix choice exactly.
The recursion \eqref{eq:state-recursion} preserves the matrix order and
recovers every coordinate by \eqref{eq:state-product}. Its width is fixed
and its depth is $O(n)$. The global patching and the finite hat
decomposition \eqref{eq:hat-decomposition} extend the construction from a
localized atom to the prescribed seed without changing these asymptotic
bounds. All constants depend only on $p$, $J$, and the fixed seed
decomposition. This proves every assertion of the theorem.
\end{proof}

\subsection{Depth Optimality}
\label{sec:obstruction}

\begin{proof}[Proof of Proposition~\ref{thm:iclr-depth-lower}]
Let
\[
 g(x)=\max\{1-|2x-1|,0\},
\qquad
 a_0^{[k]}=1,\qquad a_1^{[k]}=-\gamma_k,
\]
where $\gamma_k=1$ for odd $k$ and $\gamma_k=2$ for even $k$.  These are
bounded masks with common support $\{0,1\}$, and the sequence is genuinely
nonstationary.  Since $g$ is supported on $[0,1]$ and vanishes at both
endpoints,
\[
 (V_kf)(x)=f(2x)-\gamma_k f(2x-1)
\]
places two nonoverlapping, nonzero scaled copies of $f$ on the left and right
halves of $[0,1]$.  Induction shows that on every dyadic cell
$[m2^{-n},(m+1)2^{-n}]$, the function $F_n$ is a nonzero multiple of
$g(2^nx-m)$.  Each of the $2^n$ cell midpoints is therefore a genuine
breakpoint.  Hence $F_n$ has at least $2^n$ maximal affine pieces.

It remains to count pieces in a one-input ReLU network.  Suppose the joint
output of one layer is affine on each of $N$ intervals.  On each interval,
every one of the at most $W$ preactivations in the next layer is affine and
has at most one zero unless it vanishes identically.  Their union therefore
divides that interval into at most $W+1$ subintervals.  ReLU does not create
any other breakpoint.  Starting with the affine input and iterating over $L$
ReLU layers, the final affine output has at most $(W+1)^L$ affine pieces.
Exact equality with $F_n$ forces
$(W+1)^L\geq2^n$, which is
equation~\eqref{eq:iclr-depth-lower}.  At fixed width the dense parameter
count is at least a constant multiple of the number of nontrivial layers, so
it is also $\Omega(n)$.
\end{proof}

\subsection{Applications}
\label{sec:applications}

Algorithm~\ref{alg:certified-cascade} and
Theorem~\ref{thm:iclr-exact} turn a prescribed refinement rule into a
finite evaluator with a uniform error certificate. The following uses
require known masks and the uniform bounds specified in the algorithm;
they do not involve training a network.

\paragraph{Spline synthesis}
Consider the two-channel exponential-spline masks in
Proposition~\ref{prop:exponential-spline-certificate}.
The input consists of fixed exponents $\lambda_1,\lambda_2$, the hat seed
$g=(h,h)^T$, and a tolerance $\eps>0$. The formulas in that proposition
specify both the masks and the matching maps. The requested output is an
evaluator for their limiting profile $\Phi$, not merely values on a fixed
grid. This is a concrete nonstationary generator-synthesis problem within
the background of exponential-spline subdivision
\cite{JeongLeeYoon2013}.

Fix $q\in(1/2,1)$ and let a certified constant $D_\lambda$ satisfy
$\delta_k\leq D_\lambda 2^{-k}$ for every $k\geq1$.
Together with a certified tail constant $C_q$ from
Theorem~\ref{thm:iclr-computable-cascade}, this gives
\[
 \sum_{j=1}^n\delta_jq^{n-j}
 \leq \frac{D_\lambda}{2q-1}q^n,
 \qquad
 \sum_{j>n}\delta_j\leq D_\lambda 2^{-n}\leq D_\lambda q^n.
\]
Thus a directly usable tolerance bound is
\begin{equation}
 \|F_n-\Phi\|_\infty\leq\widetilde C_q q^n,
 \qquad
 \widetilde C_q:=C_q\left(1+\frac{D_\lambda}{2q-1}+D_\lambda\right).
 \label{eq:application-spline-tail}
\end{equation}
For example, it suffices to select
\begin{equation}
 n_\eps=\max\left\{1,
 \left\lceil\frac{\log(\widetilde C_q/\eps)}{\log(q^{-1})}\right\rceil
 \right\}.
 \label{eq:application-depth}
\end{equation}
The algorithm constructs the frames and checks the defect envelope at the
required levels, then compiles the ordered cascade into a network
$\mathcal N_\eps=F_{n_\eps}$. It returns
$\|\mathcal N_\eps-\Phi\|_\infty\leq\eps$.
The geometric-series estimate above is a consequence of the existing tail
theorem; it is not a separate convergence assumption.

The constants in \eqref{eq:application-depth} must include the finite
initial levels and an enclosure of the infinite tail. The asymptotic
statements $\rho_{\rm ch}\leq1/2$ and $\delta_k=O(2^{-k})$ establish the
rate but do not, by themselves, supply numerical stopping constants.
For this reason, the SVD calculations and the convergence plot in
Section~\ref{sec:iclr-diagnostics} illustrate the construction rather than
replace its spectral certificate. The application uses convergence of the
specified masks; it does not assert additional interpolation or exact
exponential-reproduction properties for their normalization.

\paragraph{Network evaluation}
Once the constants are supplied, \eqref{eq:application-depth} gives
$n_\eps=O(1+\log\eps^{-1})$ as $\eps\downarrow0$.
The resulting network has fixed width and
$O(1+\log\eps^{-1})$ depth and affine parameters. Its weights are
constructed from the masks, not fitted to samples. At fixed width, each
query of the generator requires $O(n_\eps)$ real-arithmetic operations.
This is useful when the same profile is evaluated repeatedly at selected
locations: the representation executes the refinement rule instead of
storing all pieces of its fine-resolution output. It extends the
refinable-function representation viewpoint of
\cite{DaubechiesEtAl2023} to the ordered matrix setting.
The count concerns real arithmetic and representation size, not uniformly
bounded weights, floating-point bit complexity, or a measured speedup over
an optimized subdivision implementation. Producing many output samples
still requires evaluating those samples.

\paragraph{Signal synthesis}
The generator error also controls finite reconstructions with known data.
For prescribed shifts $\xi_1,\ldots,\xi_M$ and coefficient vectors
$b_1,\ldots,b_M\in\R^p$, define
\[
 u(x)=\sum_{i=1}^M b_i^T\Phi(x-\xi_i),
 \qquad
 u_n(x)=\sum_{i=1}^M b_i^TF_n(x-\xi_i),
 \qquad A_{\rm syn}=\sum_{i=1}^M\|b_i\|_1.
\]
Translation invariance of the uniform norm and the triangle inequality give
\begin{equation}
 \|u-u_n\|_\infty
 \leq A_{\rm syn}\|\Phi-F_n\|_\infty
 \leq A_{\rm syn}\widetilde C_q q^n
 \label{eq:application-synthesis-error}
\end{equation}
for the spline family above. For other families, the same first inequality
applies with the general tail bound \eqref{eq:iclr-computable-tail}.
If $A_{\rm syn}>0$, a target signal error $\eps$ is therefore obtained by
calling Algorithm~\ref{alg:certified-cascade} with generator tolerance
$\eps/A_{\rm syn}$; if $A_{\rm syn}=0$, the signal is zero.
Translated copies of the compiled generator and a final linear
combination evaluate $u_n$. For fixed $M$, this preserves fixed width and
linear depth in the truncation level, with constants allowed to depend
on $M$. The application is evaluation of an already specified synthesis
formula. Recovering its coefficients or shifts from samples, proving a
stable basis property, and controlling derivative errors require
additional results and are not asserted here.

\section{Metric Entropy}
\label{sec:iclr-bits}

A second application is storage or transmission of known multiscale
profiles. Instead of storing all sampled values on a fine grid, an encoder
quantizes the generator data and a shared decoder reconstructs the profile.
Here this task is studied for the explicitly constructed two-channel
classes below: the encoder is given their source coefficients, not an
arbitrary observed signal. Metric entropy then identifies the minimum
fixed-length description needed for a prescribed worst-case error.

For a compact function class $\mathcal C$, let $N(\eps,\mathcal C)$ be its
covering number by closed uniform-norm balls of radius $\eps$, with
arbitrary bounded-function centers, and set
$H_\eps(\mathcal C)=\log_2N(\eps,\mathcal C)$.
A deterministic fixed-length code of worst-case distortion $\eps$ needs
at least $\lceil H_\eps(\mathcal C)\rceil$ bits.
The decoder, contraction data, and requested accuracy are shared
information, not charged per target. We identify this budget inside an
exactly matched cascade class and, for the stated rational data, construct
finite rational ReLU reconstructions attaining it. The result concerns
transmitted generator bits, not a listing of every assembled network
parameter or a complete file format. It is not a universal compression
claim for arbitrary signals. In the subsection titles, fixed and variable
rates refer to the contraction factors across levels; the source codes
remain fixed-length at each requested accuracy.

\subsection{Fixed Rates}
\label{sec:sharp-coding}

\begin{theorem}
\label{thm:stable-entropy}
Fix $a\in(1/2,1)$, $h(x)=(1-|x-1|)_+$, $g=(0,h)^T$, and
$t=(t_k)_{k\ge1}\in[0,1]^\N$.  Let
\begin{equation}
 A_0^{[k]}=\begin{pmatrix}a&0\\0&1/2\end{pmatrix},\qquad
 A_1^{[k]}=\begin{pmatrix}0&0\\0&1\end{pmatrix},\qquad
 A_2^{[k]}=\begin{pmatrix}0&t_k\\0&1/2\end{pmatrix}.
\label{eq:stable-source-masks}
\end{equation}
Then, uniformly in $t$,
\begin{enumerate}[label=(\roman*),leftmargin=1.7em,itemsep=2pt,topsep=2pt]
\item
\begin{equation}
 \Phi_t=(\phi_t,h)^T,\qquad
 \phi_t(x)=\sum_{k\ge1}a^{k-1}t_kh(2^kx-2),\qquad
 \|V_1\cdots V_ng-\Phi_t\|_\infty\le a^n,
\label{eq:stable-source-expansion}
\end{equation}
with $\delta_k=0$ and $\rho_{\rm ch}=a$ for $P=(0,1,0,1)$.
\item For $\mathcal C_a:=\{\phi_t:t\in[0,1]^\N\}$,
\begin{equation}
 \|\phi_t-\phi_s\|_\infty
 =\sup_{k\ge1}a^{k-1}|t_k-s_k|,
\label{eq:stable-source-isometry}
\end{equation}
and
\begin{equation}
 H_\eps(\mathcal C_a)
 =\frac{(\log_2\eps^{-1})^2}{2\log_2 a^{-1}}
 +O_a(\log\eps^{-1}),\qquad \eps\downarrow0.
\label{eq:stable-source-entropy}
\end{equation}
The vector class $\{\Phi_t\}$ has the same entropy.
\item If $a\in\mathbb Q$, there is a dyadic code with distortion $\le\eps$ and
\begin{equation}
 B_\eps\le
 \frac{(\log_2\eps^{-1})^2}{2\log_2 a^{-1}}
 +O_a(\log\eps^{-1}),
\label{eq:stable-source-code}
\end{equation}
whose rational ReLU decoders have width $O(1)$ and
$O_a(\log\eps^{-1})$ depth and parameters; the leading coefficient is
optimal among deterministic decoders.
\end{enumerate}
\end{theorem}

\begin{proof}
Let $H$ denote scalar refinement with mask $(1/2,1,1/2)$.  The elementary
hat identity $Hh=h$ gives
\[
 V_k(u,v)(x)=\bigl(a u(2x)+t_kv(2x-2),\,Hv(x)\bigr)^T.
\]
Induction, with the rightmost operator acting first, shows that the first
coordinate of $V_1\cdots V_ng$ is the first $n$ terms of
equation~\eqref{eq:stable-source-expansion}, and its second coordinate is
$h$.  The $k$th hat is supported on
$[2^{1-k},2^{2-k}]$ and attains one at $x_k=3\,2^{-k}$.
These intervals have disjoint interiors, and all hats vanish at their shared
endpoints.  The truncation error is consequently at most $a^n$, and the
uniform limit is continuous, including at the accumulation point zero.
Evaluation at $x_k$ and the disjoint interiors give
equation~\eqref{eq:stable-source-isometry}.

We next verify that this source belongs to the stable class, rather than
merely having a convergent explicit formula.  Its two transitions are
\[
 T_0^{[k]}=\begin{pmatrix}A_0^{[k]}&0\\A_2^{[k]}&A_1^{[k]}\end{pmatrix},
 \qquad
 T_1^{[k]}=\begin{pmatrix}A_1^{[k]}&A_0^{[k]}\\0&A_2^{[k]}\end{pmatrix}.
\]
Direct multiplication gives $PT_\epsilon^{[k]}=P$.  Since $Hh=h$, the
seed quotient defect also vanishes.  In fixed coordinates
$(u_0,u_1,z)$ on $\ker P$, represented in the original block order by
$(u_0,z,u_1,-z)$, the restrictions are
\begin{equation}
 C_{k,0}=\begin{pmatrix}a&0&0\\0&0&t_k\\0&0&1/2\end{pmatrix},\qquad
 C_{k,1}=\begin{pmatrix}0&a&0\\0&0&-t_k\\0&0&1/2\end{pmatrix}.
\label{eq:stable-source-restrictions}
\end{equation}
Their upper-left two-dimensional blocks have infinity norm at most $a$,
and the upper-right columns have norm at most one.  Every product of
length $m$ therefore has diagonal blocks bounded by $a^m$ and $2^{-m}$,
and its off-diagonal block is bounded by
\[
 \sum_{j=0}^{m-1}a^j2^{-(m-1-j)}
 =\frac{a^m-2^{-m}}{a-1/2}.
\]
This gives a uniform $O_a(a^m)$ product bound.  Conversely, the all-zero
digit product maps $(1,0,0)^T$ to $a^m(1,0,0)^T$, so
$\rho_{\rm ch}=a$.  Changing to the orthonormal frame used in
Theorem~\ref{thm:iclr-computable-cascade} does not change this radius.
The fixed matching matrix has singular value $\sqrt2$, verifying all of
that theorem's matching hypotheses.

Compactness follows from the compact product cube and the uniformly
vanishing coordinate tails.  Appending the same coordinate $h$ preserves
all distances.  The entropy assertion follows from the exact covering
calculation in Theorem~\ref{thm:exact-source-code} below, whose proof uses
only the established isometry.

For the dyadic code, set $w_k=a^{k-1}$, $X=\log_2\eps^{-1}$,
$\lambda=\log_2a^{-1}$, and $m=\lceil X/\lambda\rceil$, and set
$\widehat t_k=0$ for $k>m$.  For $1\leq k\leq m$, use
\[
 b_k=\left\lceil\log_2\frac{w_k}{\eps}\right\rceil,
 \qquad
 \widehat t_k=2^{-b_k}\lfloor2^{b_k}t_k\rfloor.
\]
There are $2^{b_k}+1$ possible values, including one, so each coordinate
can be transmitted using $b_k+1$ bits.  The retained coordinates have
$w_k|t_k-\widehat t_k|\leq\eps$, and the discarded coordinates have
weight at most $a^m\leq\eps$.  Equation~\eqref{eq:stable-source-isometry}
gives distortion at most $\eps$, without summing errors over the scales.
Moreover,
\[
 \sum_{k=1}^m(b_k+1)
 \leq mX-\frac{\lambda m(m-1)}2+2m
 =\frac{X^2}{2\lambda}+O_a(X).
\]
This proves the stated dyadic code length.  Finally, the
decoded first $m$ masks and seed are rational when $a$ is rational.
Theorem~\ref{thm:iclr-exact} then supplies the claimed exact network for
the quantized finite cascade.  The fixed decoder regenerates its repeated
coefficients; the leading coefficient above refers to the transmitted
generator bits, not to a listing of all assembled network parameters.
\end{proof}

Entropy from decaying coordinates belongs to the classical framework
\cite{KolmogorovTikhomirov1959,AllardBolcskei2024Entropy}.  In particular,
the exact hyperrectangle covering identity is established in
\cite[Thm.~18]{AllardBolcskei2026Exact}.  We apply this identity to the
weighted-coordinate geometry above.  For rational $a$, its optimal covering
centers can be chosen as finite rational cascades and hence finite ReLU
networks at every fixed accuracy.  This construction uses triangular
matrices and does not require irreducibility.

\begin{theorem}
\label{thm:exact-source-code}
For $\mathcal C_a$ in Theorem~\ref{thm:stable-entropy}, set
\[
 w_k=a^{k-1},\qquad n_k=\left\lceil\frac{w_k}{2\eps}\right\rceil.
\]
Then, for every $\eps>0$,
\begin{equation}
 N(\eps,\mathcal C_a)=\prod_{k\ge1}n_k,
 \qquad
 B_\eps^*=\left\lceil\sum_{k\ge1}\log_2 n_k\right\rceil,
\label{eq:exact-cover-code}
\end{equation}
where $B_\eps^*$ is the minimum deterministic fixed-length budget and only
finitely many $n_k$ exceed $1$.  If $a\in\mathbb Q$, $B_\eps^*$ is attained by a
mixed-radix code with rational ReLU reconstructions of width $O(1)$ and
$O_a(1+\log_+\eps^{-1})$ depth and parameters, where
$\log_+u:=\max\{\log_2u,0\}$.

For $X:=\log_2\eps^{-1}>1$ and $\lambda:=\log_2a^{-1}$,
\begin{equation}
 H_\eps(\mathcal C_a)
 =\frac{X^2}{2\lambda}
 +\left(\frac12-\frac1\lambda\right)X+O_a(1),
\label{eq:second-order-entropy}
\end{equation}
and
\begin{equation}
 0\leq H_\eps(\mathcal C_a)
       -\frac{(X-1)^2}{2\lambda}-\frac{X-1}{2}
 \leq \frac{\lambda}{8}+\frac{1}{(1-a)\ln2}.
\label{eq:entropy-remainder}
\end{equation}
\end{theorem}

\begin{proof}
The covering argument applies to any positive decreasing weights tending
to zero in \eqref{eq:stable-source-isometry}.  For each $n_k\geq2$, take
the $n_k$ points $j/(n_k-1)$, $0\leq j<n_k$, and set the other coordinates
to zero.  Their Cartesian product has $\prod_k n_k$ elements.  Distinct
elements differ in some coordinate by weighted distance at least
$w_k/(n_k-1)>2\eps$.  No radius-$\eps$ ball, even with an ambient center,
can contain two of them.

For attainment, let $m=\min\{j\geq0:w_{j+1}\leq\eps\}$ and set
\begin{equation}
 j_k=\min\{\lfloor n_kt_k\rfloor,n_k-1\},\qquad
 \widehat t_k=\frac{2j_k+1}{2n_k}\quad(k\leq m),\qquad
 \widehat t_k=0\quad(k>m).
\label{eq:midpoint-code}
\end{equation}
Retained coordinates have error at most $w_k/(2n_k)\leq\eps$.
In particular, the coordinates with $\eps<w_k\leq2\eps$ use the shared
midpoint $1/2$ and cost no bits.  The discarded tail has weight at most
$\eps$.  The isometry therefore gives distortion at most $\eps$ with
exactly $\prod_k n_k$ finite-support centers.  Ranking the digits $j_k$
in their mixed radices gives \eqref{eq:exact-cover-code}, including a
zero-bit code when the product is one.  For rational $a$, these centers
are rational finite cascades of length $m=O_a(1+\log_+\eps^{-1})$.
The exact compiler gives the claimed networks.  For an algorithm using
exact comparisons, the shared accuracy may be taken rational.  For arbitrary
real accuracy the codebook assertion is existential, with the finitely
many integers $m,n_k$ fixed as shared decoder data.

To obtain the remainder, set $Y=X-1$, $M=\lceil Y/\lambda\rceil$, and
$\theta=M-Y/\lambda\in[0,1)$.  The active ratios are
$r_j=2^{Y-j\lambda}>1$, $0\leq j<M$.  Consequently
\[
 H_\eps=MY-\frac{\lambda M(M-1)}2+R_\eps
 =\frac{Y^2}{2\lambda}+\frac Y2
    +\frac\lambda2\theta(1-\theta)+R_\eps.
\]
Since $0\leq\log_2\lceil r\rceil-\log_2r
\leq1/(r\ln2)$ and $\sum_{j=0}^{M-1}r_j^{-1}\leq(1-a)^{-1}$,
we have $0\leq R_\eps\leq((1-a)\ln2)^{-1}$.  This proves both
\eqref{eq:second-order-entropy} and \eqref{eq:entropy-remainder}.
\end{proof}

The exact budget excludes headers and assembled-parameter listings.
The dyadic code above attains the leading term but can exceed the exact
mixed-radix budget by $O_a(X)$ bits.

\subsection{Variable Rates}
\label{sec:variable-rates}

The fixed contraction factor can be replaced by a prescribed clock of
contraction rates.  This extension distinguishes uniform stability over all
starting levels from the accumulation of visible scales at the initial level.

\begin{theorem}
\label{thm:variable-rate-source}
Let $\boldsymbol a=(a_k)_{k\ge1}$ satisfy
$1/2<\underline a\le a_k\le\overline a<1$.  Replace $a$ by $a_k$ in
\eqref{eq:stable-source-masks} and set
\[
 w_k:=\prod_{i=1}^{k-1}a_i,\qquad
 L_j:=\sum_{i=1}^j\log_2a_i^{-1},\qquad L_0:=0,
\]
\[
 \phi_t(x):=\sum_{k\ge1}w_kt_kh(2^kx-2),\qquad
 \Phi_t:=(\phi_t,h)^T,\qquad
 \mathcal C_{\boldsymbol a}:=\{\phi_t:t\in[0,1]^\N\}.
\]
Then
\begin{equation}
 \|V_1\cdots V_ng-\Phi_t\|_\infty\le w_{n+1},\qquad
 \|\phi_t-\phi_s\|_\infty=\sup_k w_k|t_k-s_k|,
\label{eq:variable-isometry}
\end{equation}
and
\begin{equation}
 \rho_{\rm ch}=2^{-\lambda_{\rm u}},\qquad
 \lambda_{\rm u}:=\lim_{m\to\infty}\inf_{r\ge1}
 \frac1m\sum_{i=r}^{r+m-1}\log_2a_i^{-1},
 \qquad
 \rho_{\rm ch}\le\overline a<1.
\label{eq:variable-radius}
\end{equation}
For $X:=\log_2\eps^{-1}>1$,
\begin{equation}
 H_\eps(\mathcal C_{\boldsymbol a})
 =\sum_{j\ge0}(X-1-L_j)_++R(X),\qquad
 0\le R(X)\le\frac1{(1-\overline a)\ln2}.
\label{eq:variable-entropy-sum}
\end{equation}
Moreover,
\begin{equation}
 \frac{L_n}{n}\to\lambda
 \quad\Longrightarrow\quad
 H_\eps(\mathcal C_{\boldsymbol a})
 =\frac{X^2}{2\lambda}+o(X^2),
\label{eq:variable-entropy-leading}
\end{equation}
with remainder $O(X)$ if $\sup_n|L_n-\lambda n|<\infty$.  If
$\boldsymbol a\subset\mathbb Q$ is generated by a fixed algorithm, then at shared
dyadic accuracies the optimal budget $\lceil H_\eps\rceil$ is attained by
rational ReLU decoders of width $O(1)$ and depth/parameter count $O(X)$.
\end{theorem}
The shared clock description is not charged per target; no bound is asserted
for decoder running time or for an expanded listing of its rational parameters.

\begin{proof}
The refinement identity, matching calculation, and disjoint hat supports in
Theorem~\ref{thm:stable-entropy} apply with $a$ replaced by $a_k$.
They give \eqref{eq:variable-isometry}, continuity at zero, and compactness
of the source class.  The restrictions are
\eqref{eq:stable-source-restrictions} with the same replacement.
For a product starting at $r$ of length $m$, write
$A_{r,m}=\prod_{i=r}^{r+m-1}a_i$.
The upper-left block has norm at most $A_{r,m}$, the lower-right entry is
$2^{-m}$, and the off-diagonal block is bounded by
\[
 \sum_{j=0}^{m-1}\left(\prod_{i=r}^{r+j-1}a_i\right)2^{-(m-1-j)}
 \leq\frac{A_{r,m}}{\underline a-1/2}.
\]
The all-zero digit product sends the first basis vector to $A_{r,m}$
times itself.  Thus, in these fixed coordinates,
\[
 p_m\leq\chi_m\leq
 \left(1+\frac1{\underline a-1/2}\right)p_m,
 \qquad p_m=\sup_r A_{r,m}.
\]
The infima of the length-$m$ logarithmic sums in
\eqref{eq:variable-radius} are superadditive.  Their normalized limit
therefore exists, and taking $m$th roots proves the radius formula.

The covering and finite-center construction of
Theorem~\ref{thm:exact-source-code} uses only the weighted distance identity
and decreasing weights tending to zero.  Consequently,
\[
 N(\eps,\mathcal C_{\boldsymbol a})
 =\prod_{k\geq1}\left\lceil\frac{w_k}{2\eps}\right\rceil.
\]
For $y>1$, $0\leq\log_2\lceil y\rceil-\log_2y
\leq1/(y\ln2)$.  The reciprocals $2\eps/w_k$ over the active coordinates
$w_k>2\eps$, read backwards, decrease at least geometrically with ratio
$\overline a$.  Their sum is at most $1/(1-\overline a)$, proving
\eqref{eq:variable-entropy-sum}.
If $L_j/j\to\lambda$, comparison with
$(\lambda\pm\eta)j$ outside a fixed prefix, followed by $\eta\downarrow0$,
gives
$\sum_j(Y-L_j)_+=Y^2/(2\lambda)+o(Y^2)$.
Bounded discrepancy changes this sum from
$\sum_j(Y-\lambda j)_+$ by $O(Y)$, since only $O(Y)$ terms can be
nonzero.  Taking $Y=X-1$ proves both asymptotic assertions.
Finally, the exact-cover encoder retains only weights $w_k>\eps$,
so its rational midpoint masks have $O(X)$ levels.  A computable rational
clock lets the shared decoder recover those masks and the mixed-radix
alphabet sizes exactly; the finite compiler completes the construction.
\end{proof}

\begin{corollary}
\label{cor:radius-not-entropy}
There exist computable rational clocks $\boldsymbol a^{(1)},\boldsymbol a^{(2)}$
such that, with $X=\log_2\eps^{-1}$,
\[
 \rho_{\rm ch}^{(1)}=\rho_{\rm ch}^{(2)}=\frac9{10},
\]
but
\[
 \lim_{\eps\downarrow0}\frac{H_\eps(\mathcal C_{\boldsymbol a^{(1)}})}{X^2}
 =\frac1{2\log_2(10/9)}
 \neq
 \frac1{2\log_2(5/3)}
 =\lim_{\eps\downarrow0}\frac{H_\eps(\mathcal C_{\boldsymbol a^{(2)}})}{X^2}.
\]
\end{corollary}

\begin{proof}
Take first the constant clock $a_k=9/10$.  For the second, set $a_k=3/5$
except on the blocks
\[
 \{2^{j^2},\ldots,2^{j^2}+j-1\},\qquad j=1,2,\ldots,
\]
where $a_k=9/10$.  These disjoint blocks have unbounded lengths and zero
asymptotic density: up to a level $n$ their total length is $O(\log(n+1))$.
Every window length therefore occurs inside a high-rate block, giving
$p_m=(9/10)^m$ and $\rho_{\rm ch}=9/10$ for both clocks.
The second clock nevertheless has
$L_n/n\to\log_2(5/3)$, which proves the stated entropy coefficients by
\eqref{eq:variable-entropy-leading}.
\end{proof}

Thus the coding coefficient measures the prefix accumulation of logarithmic
contraction, whereas the chronological radius measures the least contraction
over arbitrarily placed long windows.  The distinction occurs within the
same bounded, zero-defect, two-channel cascade construction.

\section{Numerical Validation}
\label{sec:iclr-diagnostics}

The experiments examine two uses of the theory: evaluation of prescribed
refinement generators, as in Section~\ref{sec:applications}, and
finite-accuracy storage of the source classes in Section~\ref{sec:iclr-bits}.
The stability tests measure the effect of product order, frame construction,
and truncation; their exponential-spline example supplies the generator
used in the application discussion. The fixed- and variable-rate tests
implement encoding and decoding and check their bit budgets and errors.
These are synthetic, known-generator experiments, not benchmarks for
learned masks or general-purpose signal compression. In particular, the
spline test measures errors at fixed depths; it does not report certified
stopping depths obtained from numerical enclosures of the constants in
\eqref{eq:application-depth}.

\subsection{Stability Tests}

\paragraph{Setup}
These experiments use synthetic data to examine product order, convergence,
moving matching spaces, and quantization.  Panel (a) instantiates the sharpness example in
Section~\ref{sec:iclr-limit} with $q=2/3$ and
$\delta_j=0.30(0.45)^{j-1}$, using 400 shuffled orders per depth.  Panel (b)
uses the scalar alternating family
\begin{equation}
 a^{[k]}=(\beta_k+t_k,\ 1-t_k,\ 1-\beta_k).
\label{eq:iclr-alternating}
\end{equation}
Here $\beta_k$ alternates between $1/3$ and $2/3$ and
$t_k=2^{-k-4}$.  Its restricted JSR is $2/3$ and its alias defect is
$2|t_k|$.  We compare references $F_{18},F_{20},F_{22}$ on grids
$2^{-16},2^{-17},2^{-18}$ and plot the unscaled right-hand-side profile
$B_n(0.70)$ defined below. Panel (c) uses the natural
two-channel exponential-spline mask of
Proposition~\ref{prop:exponential-spline-certificate}, with exponents
$1$ and $-3/4$.  At every level it constructs $P_k$, $S_k$, and the corrected
transitions by SVD and pseudoinverse; no frame is entered by hand.  Panel (d)
compares mask-tied and independent-entry
stochastic dyadic quantization at depth 12 over 80 fixed-seed trials.  This
last comparison diagnoses algebraic tying, not a universal separation
between network encoders.
\begin{equation}
 B_n(q)=q^n+\sum_{j=1}^n2|t_j|q^{n-j}+\sum_{j>n}2|t_j|.
\label{eq:iclr-tail}
\end{equation}

\begin{figure}[htbp]
\centering
\includegraphics[width=\linewidth]{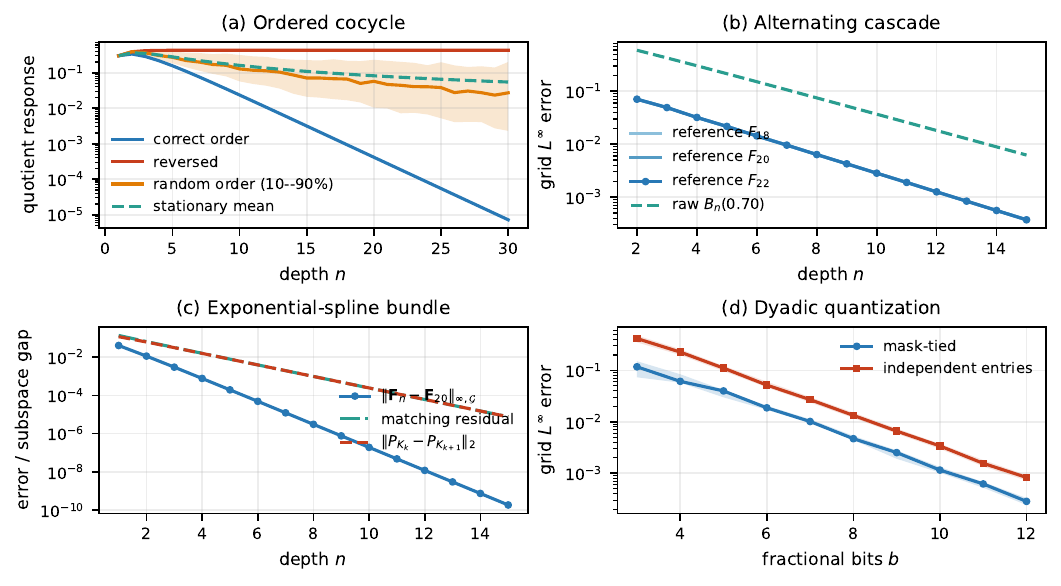}
\caption{Mechanism diagnostics.  (a) Reversal, random shuffling, and a
stationary mean change the ordered defect response.  (b) Three finite
references are visually indistinguishable; the dashed curve is the raw, not
fitted, profile $B_n(0.70)$ in equation~\eqref{eq:iclr-tail}.  (c) A
noncommuting exponential-spline cocycle converges while the computed matching
residual and successive-space gap decay.  (d) Median grid error and
interquartile bands over 80 trials.}
\label{fig:mechanism-experiments}
\end{figure}

\paragraph{Results}
In Figure~\ref{fig:mechanism-experiments}(a), the correct quotient
response falls from $0.30$ to $7.22\times10^{-6}$ by depth 30, while reversal
approaches $0.4286$.  The responses differ by nearly five orders of magnitude,
illustrating the effect of temporal order.  In panel (b), a
log-linear fit over depths 7--14 gives factor $0.6662$, consistent with the
restricted JSR $2/3$.  Changing the reference from $F_{18}$ to $F_{22}$ alters
the errors by at most $0.024\%$, while grid refinement changes them by at most
$0.35\%$; moreover $\max_n\|F_n-F_{22}\|_{\infty,\mathcal G}/B_n(0.70)=0.119$
 without fitting a multiplier.  In panel (c), the mixed $4\times4$ transition
commutator has norm $0.5464$.  The automatically constructed frames have
condition number one.  The theoretical corrected cluster JSR is $1/2$,
and both the computed matching residual and
$\|P_{K_k}-P_{K_{k+1}}\|_2$ halve asymptotically.  The measured cascade tail
decreases by a factor $0.25$ to $1.88\times10^{-10}$ at depth 15.  Panel (d) gives a
$2.5$--$3.7\times$ median reduction from mask tying.  These observations
describe the prescribed examples and their finite-grid errors.  The
spectral value is supplied by the analytical calculation, while the
experiments measure the frame conditioning, residuals, and output errors.
The quantization comparison concerns shared mask data rather than training
or a worst-case separation between network encoders.

\subsection{Fixed-Rate Tests}
\label{sec:numerical-coding}

\paragraph{Setup}
We test whether the finite-support mixed-radix construction in
Theorem~\ref{thm:exact-source-code} attains the exact covering budget
in an end-to-end implementation, and measure its finite-accuracy gain
over the dyadic construction. We use $a\in\{3/5,3/4,9/10\}$ and
\[
 X\in\{8,16,24,32,48,64,80,96,128\},\qquad \eps=2^{-X}.
\]
For each $a$, we draw 256 independent coefficient prefixes from a uniform
dyadic grid with seed 20260911 and append a known unit tail.  The same source
is encoded at every accuracy by three codecs: the optimal midpoint
mixed-radix code of \eqref{eq:midpoint-code}, the adaptive dyadic code in
Theorem~\ref{thm:stable-entropy}, and a baseline assigning $X$ fractional
bits to every retained level.  Four common deterministic sources exercise
the endpoints and cell boundaries at each accuracy.  All methods use the
same truncation depth.  We compare logical coefficient payloads, excluding
the shared values of $a$ and $\eps$, decoder software, file metadata, and byte
padding.  This is the operational model of
Section~\ref{sec:iclr-bits}.  The mixed-radix byte stream has no header,
whereas the two reusable NSC1 baseline implementations are self-describing.
Their complete file lengths are therefore not compared here.

Every source is serialized and independently decoded.  Integer and rational
arithmetic controls planning, quantization, and the uniform error
\eqref{eq:stable-source-isometry}, including the infinite tail.  Thus a
spatial sampling grid is not used to certify distortion.  The exact integer
covering cardinality is converted to an 80-digit decimal logarithm only for
plotting the entropy remainder.  The existing rational width-five ReLU
decoder is separately evaluated at the retained peak locations and boundary
points.

\begin{figure}[htbp]
\centering
\includegraphics[width=0.99\linewidth]{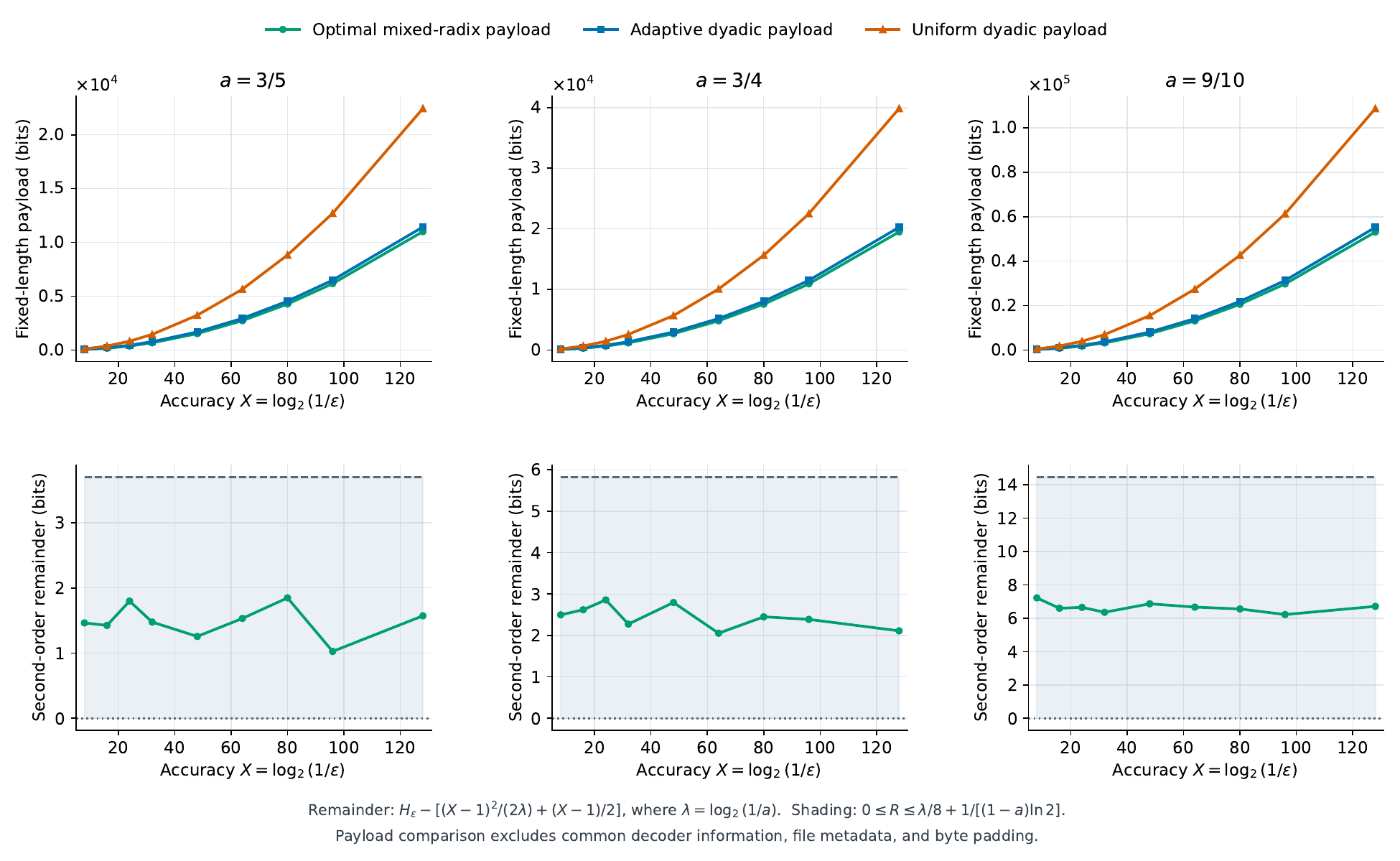}
\caption{Exact finite-accuracy coding experiment.  Top: logical payloads of
the optimal mixed-radix, adaptive dyadic, and uniform dyadic codecs on the
same retained source coordinates.  Bottom: the exact entropy after removing
the quadratic and linear terms in \eqref{eq:second-order-entropy}.  The shaded
region is the deterministic bound \eqref{eq:entropy-remainder}.  The payload
comparison excludes shared decoder information, metadata, and byte padding.}
\label{fig:exact-coding-experiment}
\end{figure}

\paragraph{Results}
Figure~\ref{fig:exact-coding-experiment} reports the implemented payloads and
the entropy remainders from Theorem~\ref{thm:exact-source-code}.
The mixed-radix payload is
$\lceil H_\eps\rceil$ at every tested point, and the measured second-order
remainder remains inside its explicit interval.  At $X=128$, the adaptive
dyadic code is only $3.91\%$--$3.94\%$ above the exact optimum, while the
uniform allocation uses slightly more than twice the optimal payload.
Table~\ref{tab:exact-coding-experiment} reports the integer budgets and the
unrounded entropy remainder.  The observed whole-class bound equals
$\eps$, which is expected because the discarded tail can approach its unit
coefficient.  The median random-source errors range from $0.9956\eps$ to
$0.9991\eps$ at $X=128$.

\begin{table}[htbp]
\centering
\caption{Finite-accuracy results at $X=128$.  All three budgets count logical
coefficient payloads under a shared decoder.  Here
$R_\eps=H_\eps-(X-1)^2/(2\lambda)-(X-1)/2$, and the final column is the bound
in \eqref{eq:entropy-remainder}.}
\label{tab:exact-coding-experiment}
\footnotesize
\setlength{\tabcolsep}{4pt}
\begin{tabular}{@{}c|rrr|rrr@{}}
\toprule
$a$ & Optimal & \shortstack{Adaptive\\dyadic} & \shortstack{Uniform\\dyadic}
& \shortstack{Adaptive\\excess} & $R_\eps$ & \shortstack{Remainder\\bound} \\
\midrule
$3/5$  & $11\,008$ & $11\,438$ & $22\,446$  & $3.91\%$ & $1.5725$ & $3.6989$ \\
$3/4$  & $19\,497$ & $20\,265$ & $39\,861$  & $3.94\%$ & $2.1121$ & $5.8227$ \\
$9/10$ & $53\,126$ & $55\,220$ & $108\,747$ & $3.94\%$ & $6.7069$ & $14.4460$ \\
\bottomrule
\end{tabular}
\end{table}

The fixed-rate experiment performs 21,060 binary round-trips and exact
whole-line error checks, consisting of all three codecs applied to the 260
sources at 27 parameter pairs.  In addition, 567 independent rational ReLU
peak and boundary evaluations pass.  The integer budgets agree at all 27
points with a separately implemented exact-cover calculation.  These checks
test the construction and finite formulas, while the packing proof establishes
optimality over arbitrary deterministic decoders.

\subsection{Variable-Rate Tests}
\label{sec:numerical-variable}
We next use three prescribed rational clocks: $a_k=9/10$ at every level, the
periodic clock $(3/5,3/4,9/10)$, and the sparse clock in
Corollary~\ref{cor:radius-not-entropy}.  For each clock we compute the exact
covering product in \eqref{eq:variable-entropy-sum} at
$X\in\{8,16,32,64,128,256,512\}$ and run the optimal codec on 16 random
sources and two endpoint sources.  These 378 additional round-trips include
the exact first omitted weight in every error certificate.  We also scan a
65,540-level prefix and report maxima over the available sliding windows.
This last calculation is only a finite-prefix diagnostic and does not certify
the supremum over all starting levels in \eqref{eq:variable-radius}.

\begin{figure}[htbp]
\centering
\includegraphics[width=0.92\linewidth]{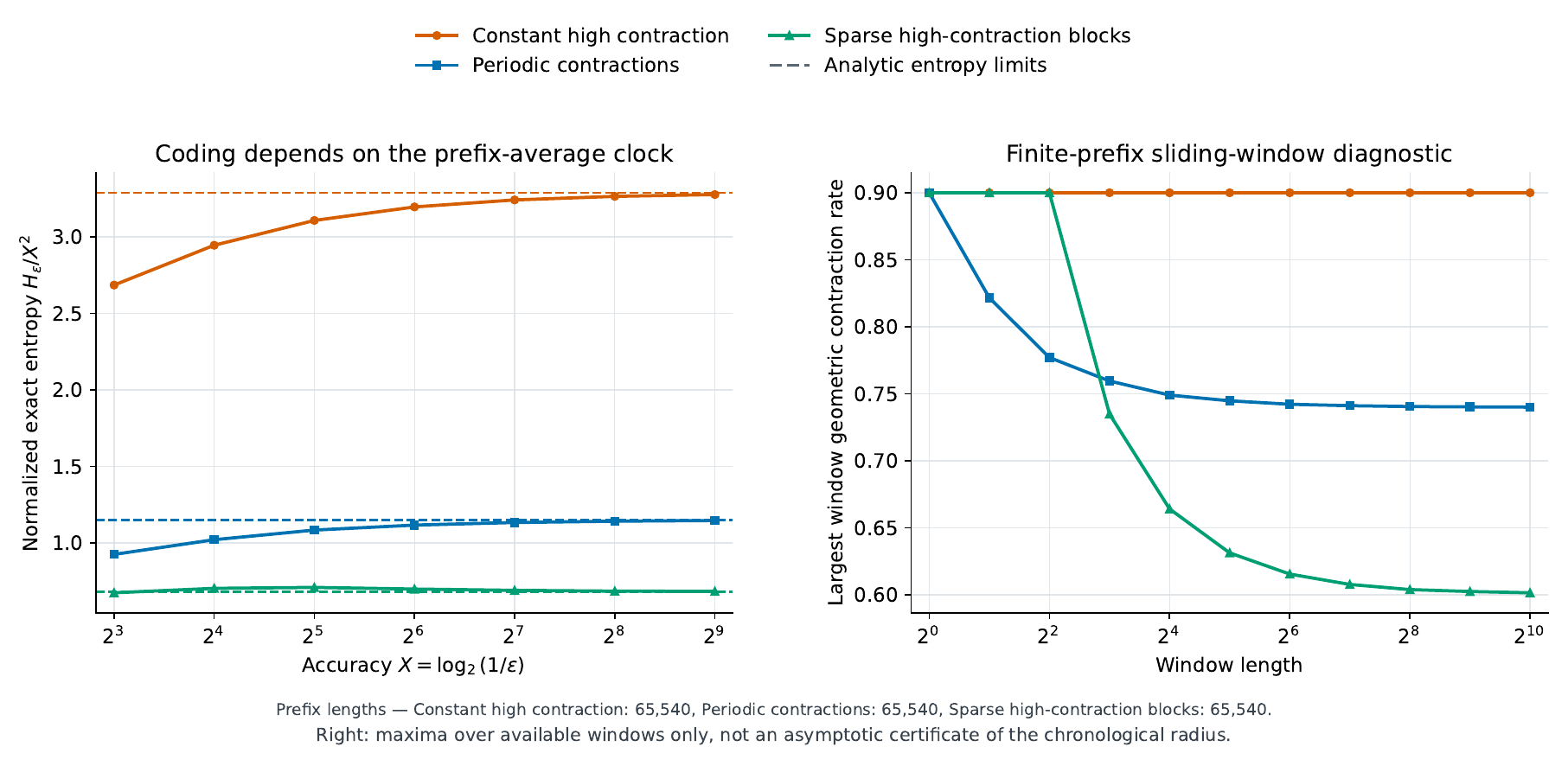}
\caption{Variable-clock experiment.  Left: exact normalized entropy, with
dashed analytic limits from \eqref{eq:variable-entropy-leading}.  The
constant-high and sparse-high clocks have different coding coefficients even
though both have chronological radius $0.9$.  Right: largest geometric
contraction rate over sliding windows contained in a 65,540-level prefix.
The sparse clock contains a high-rate block of length four in this prefix,
so long observed windows decline toward the background rate.  This finite
panel is not an asymptotic radius certificate.}
\label{fig:variable-clock-experiment}
\end{figure}

At $X=512$, $H_\eps/X^2$ is $3.277570$ for the constant-high clock and
$0.682600$ for the sparse clock, close to their respective limits $3.289407$
and $0.678458$.  Both clocks nevertheless have chronological radius $0.9$.
The periodic control gives $1.145922$ against its limit $1.150301$ and has
radius $0.739864$.  In the finite window scan, the sparse-clock maximum is
$0.9$ for length four but $0.601427$ for length 1024 because no high-rate
block that long occurs in the observed prefix.  This contrast illustrates
why the prefix average controls coding while the worst sliding windows over
the infinite clock control uniform chronological stability.

\paragraph{Conclusion}
For prescribed one-dimensional masks with common finite support,
chronological stability and exact representation give a finite-resolution
synthesis rule when a uniform spectral bound is supplied.
Section~\ref{sec:applications} makes this use explicit for
exponential-spline generators and finite signal-synthesis sums: a requested
uniform tolerance determines a truncation depth and a finite evaluator.
For the constructed two-channel classes, metric entropy also determines
an optimal finite-accuracy storage budget, which the stability radius
alone does not determine. The experiments test these mechanisms on
synthetic sources with known generators. Their sliding-window scans
describe finite prefixes, whereas infinite-clock stability and coding
optimality follow from the proofs. Mask recovery from observations,
general perfect-reconstruction filter banks, nonseparable multivariate
refinement, and efficient certification for arbitrary nonperiodic clocks
remain outside the present scope.

\section*{Acknowledgment}
\paragraph{AI disclosure}
OpenAI ChatGPT/Codex assisted with literature organization, LaTeX
formatting, and drafting and revising the background and application
discussion. The authors are responsible for the accuracy, attribution,
and originality of the submitted text, mathematical claims, and code.

\end{document}